\documentclass[a4paper,reqno]{article} 
\usepackage[T2A]{fontenc}
\usepackage[utf8]{inputenc}
\usepackage{pifont}
\usepackage[dvips]{graphicx}
\usepackage{amscd,amsmath,amssymb,amsthm,latexsym,mathrsfs,pifont}
\usepackage{longtable,cmap,indentfirst,setspace}
\usepackage{ccaption, caption, subcaption}
\usepackage{tikz,tikz-cd,array,hhline}
\usepackage{color,xcolor,hyperref}
\usepackage{verbatim}

\usepackage{geometry}
\hypersetup{
    unicode=true,          % non-Latin characters in Acrobat's bookmarks
    pdftoolbar=true,        % show Acrobat's toolbar?
    pdfmenubar=true,        % show Acrobat's menu?
    pdffitwindow=false,     % window fit to page when opened
    pdfnewwindow=true,      % links in new PDF window
    colorlinks=true,        	% false: boxed links; true: colored links
    linkcolor={red!50!black},	% color of internal links (change box color with linkbordercolor)
    citecolor={blue!50!black},  % color of links to bibliography
    urlcolor={blue!80!black},    % color of external links
    filecolor=magenta,          % color of file links
}

\DeclareMathOperator{\rk}{rk}
\DeclareMathOperator{\Pic}{Pic}
\DeclareMathOperator{\Cl}{Cl}
\DeclareMathOperator{\Hom}{Hom}
\DeclareMathOperator{\Sing}{Sing}
\DeclareMathOperator{\Image}{Im}
\DeclareMathOperator{\Ker}{Ker}

\newcommand{\XX}{\mathcal X} 
\newcommand{\FF}{\mathcal F}
\newcommand{\OO}{\mathcal O}
\newcommand{\LL}{\mathcal L}
\newcommand{\DD}{\mathcal D}
\newcommand{\A}{\mathbb A} 
\newcommand{\CC}{\mathbb C} 
\newcommand{\RR}{\mathbb R} 
\newcommand{\PP}{\mathbb P}
\newcommand{\ZZ}{\mathbb Z}

\makeatletter\@addtoreset{equation}{section} \makeatother

\theoremstyle{plain}
\newtheorem{thm}[equation]{Theorem}
\newtheorem{lem}[equation]{Lemma} 
\newtheorem{cor}[equation]{Corollary} 
\newtheorem{prop}[equation]{Proposition} 

\theoremstyle{definition}
\newtheorem{mydef}[equation]{Definition} 
\newtheorem{rem}[equation]{Remark}
\newtheorem{ex}[equation]{Example}

\title{Small toric degenerations of Fano threefolds}
\author{Sergey Galkin}
\begin{document}
%\begin{abstract}
%\end{abstract}
\date{April 1, 2008}
\maketitle
\abstract{
We classify smooth Fano threefolds that admit degenerations to toric Fano threefolds with ordinary double points.
}
\section*{Introduction}\label{sec:intro}

We consider small toric degenerations of Fano threefolds. These are the degenerations of smooth Fano threefolds to toric Fano threefolds with ordinary double points (see Definitions~\ref{def: deformation}, \ref{def: small degeneration}). These degenerations has applications in mirror symmetry. Mirror constructions for toric manifolds and complete intersections therein are given by Givental and Batyrev \cite{Ba93, Ba94, Gi}.

If a Fano manifold $Y$ admits a small toric degeneration $X$, a mirror of $Y$ can be produced via a toric construction, and it can be used to compute some Gromov--Witten invariants of $Y$.

In that way, using small toric degenerations of Grassmannians (constructed in \cite{S2}) and varieties of partial flags (constructed in \cite{GL2}), mirror images of these homogeneous varieties were constructed in \cite{BFKS1, BFKS2}. 

Generalizing these examples Batyrev introduced the notion of small toric degeneration of a Fano variety \cite{Ba97}.
The complete classification of smooth Fano threefolds is well known thanks to works of Fano, Iskovskikh, Shokurov, Mori--Mukai \cite{Is79,MM1,MM2,MM3} (also see \cite{Is88,Mu92,IP} for reviews and alternative constructions). 

Batyrev posed a very natural question \cite[Question 3.9]{Ba97}: 
\begin{quote}
\emph{Which 3-dimensional nontoric smooth Fano varieties do admit small toric degenerations?}
\end{quote}

Theorem~\ref{thm: admitting toric degeneration} of this paper provides a complete answer to this question.
Section~\ref{sec:applications} contains some application of these degenerations.

\tableofcontents

\section{A formulation of the main result.}\label{sec:claim}
\begin{mydef}\label{def: deformation}
	Deformation is a flat proper morphism
	\[
	\pi: \XX\to\Delta,
	\]
	where $\Delta$ is a unit disc $\{|t| < 1\}$, and $\XX$ is an irreducible complex manifold. All the deformations	we consider are projective ($\pi$ is a projective morphism over $\Delta$). Denote fibers of $\pi$ by $X_t$, and let $i_{t\in\Delta}$ be the inclusion of a fiber $X_t\to X$.
\end{mydef}

If all fibers $X_{t\ne 0}$ are nonsingular, then the deformation $\pi$ is called a degeneration of $X_{t\ne 0}$ or a
smoothing of $X_0$. If at least one such morphism $\pi$ exists, we say that varieties $X_{t\ne 0}$ are {\it smoothings}
of $X_0$, and $X_0$ is a {\it degeneration} of $X_{t\ne 0}$.

For a coherent sheaf $\FF$ on $\XX$ over $\Delta$ and $t\in\Delta$ the symbol $\FF_t$ stands for the restriction $i^*_t\FF$ to the fiber over $t$. In particular there is a well-defined restriction morphism on Picard groups $i^*_t: \Pic(\XX)\to\Pic(\XX_t)$.

\begin{mydef}\label{def: small degeneration}[\cite{Ba97}]
	Degeneration (or a smoothing) $\pi$ is {\it small}, if $X_0$ has at most Gorenstein terminal singularities (see \cite{KMM,YPG}), and for all $t\in\Delta$ the restriction $i^*_t: \Pic(\XX)\to\Pic(\XX_t)$ is an isomorphism.
\end{mydef}

All 3-dimensional terminal Gorenstein toric singularities are nodes i.e. ordinary double points analytically isomorphic to $(xy = zt)\subset\A^4$ (see e.g. \cite{Ni}).

\begin{mydef}\label{def: index}
	The {\it index} of a (Gorenstein) Fano variety $X$ is the highest $r>0$, such that anticanonical
	divisor class $-K_X$ is an $r$-multiple of some integer Cartier divisor class $H$:
\[
-K_X=rH.
\]
\end{mydef}
\begin{mydef}\label{def: d}
	Let $H\in\Pic(X)$ be a Cartier divisor an on $n$-dimensional variety $X$, and $D_1,\ldots,D_l$ be a base of lattice $H^{2k}(X,\ZZ)/{\rm tors}$. Define $d^k(X,H)$ as a discriminant of the quadratic form $(D_1,D_2) = (H^{n-2k}\cup D_1\cup D_2)$ on $H^{2k}(X,\ZZ)/{\rm tors}$. For a Gorenstein threefold $X$ denote by $d(X) = d^1(X,-K_X)$ the anticanonical discriminant of $X$.
\end{mydef}

If $X$ is a smooth variety and $H$ is an ample divisor, then hard Lefschetz theorem states that $d^k(X,H)$ is nonzero.

\begin{mydef}\label{def: principal invariants}
	Let $X$ be a Fano threefold. Consider Picard number $\rho=\rk\Pic(X) = \dim H^2(X)$,	half of third Betti number $b=\frac{1}{2} \dim H^3(X)$, (anticanonical) degree $\deg = (-K_X)^3$, Fano index $r$ (see def. \ref{def: index}) and (anticanonical) discriminant $d$ (see def. \ref{def: d}). Numbers $\rho$, $r$, $\deg$, $b$, $d$ form a set	of {\it principal invariants} of smooth Fano threefold.
\end{mydef}

\begin{mydef}\label{def: notation}
	We use the following notations for the (families of) smooth varieties
	\begin{itemize}
		\item $\PP^n$ --- $n$-dimensional projective space;
		\item $G(l,N)$ --- Grassmannian of $l$-dimensional linear subspaces in in $N$-dimensional space.
		\item[] families of smooth surfaces:
		\item $S_d$, $d=1,\ldots,8$ --- del Pezzo surfaces of degree $d$ and index 1;
		\item[] and families of smooth Fano threefolds:
		\item $Q$ --- a quadric in $\PP^4$;
		\item $B_4$ --- intersections of two quadrics in $\PP^5$;
		\item $B_5$ --- del Pezzo quintic threefold, that is a section of $G(2,5)$ by linear subspace of codimension 3;
		\item $V_{22}$ --- Fano threefolds of genus 12 with $\rho=1$;
		\item $W$ --- divisor of bidegree $(1, 1)$ in $\PP^2\times\PP^2$ (i.e. $\PP_{\PP^2}(T_{\PP^2})$);
%		\item $B_7$ --- a blowup of $\PP^3$ in a point (i.e. $\PP_{\PP^2}(\OO\oplus\OO(1))$);
		\item $V_{\rho.N}$ ($\rho=2,3,4$) --- families of Fano threefolds with Picard number $\rho$ and number $N$ in Mori--Mukai's tables \cite[Table 2, Table 3, Table 4]{MM2}.
	\end{itemize}
\end{mydef}

We use the standard notations for toric varieties \cite{Da78,Fu93,Ba93}: a toric variety $X$ corresponding to a fan $\Sigma$ in the space $N=\ZZ^{\dim X}$, each ample divisor $H$ on $X$ corresponds to a polytope $\Delta_H$ in the dual space $M=\Hom(N,\ZZ)$; we denote the variety $X$ by the symbols $X_\Sigma$ or $\PP(\Delta)$.

\begin{thm}\label{thm: admitting toric degeneration}
	These and only these families of non-toric smooth Fano threefolds $Y$ do admit small
	degenerations to toric Fano threefolds (in notations \ref{def: notation}):
	\begin{enumerate}
		\item 4 families with $\Pic(Y)=\ZZ$: $Q,\ B_4,\ B_5,\ V_{22}$;
		\item 16 families with $\Pic(Y)=\ZZ^2$: $V_{2.n}$, where $n = 12,\ 17,\ 19,\ 20,\ 21,\ 22,\ 23,\ 24,\ 25,\ 26,\ 27,\ 28,\\ 29,\ 30,\ 31,\ 32$;
		\item 16 families with $Pic(Y)=\ZZ^3$: $V_{3.n}$, where $n = 7,\ 10,\ 11,\ 12,\ 13,\ 14,\ 15,\ 16,\ 17,\ 18,\ 19,\ 20,\\ 21,\ 22,\ 23,\ 24$;
		\item 8 families with $Pic(Y)=\ZZ^4$: $V_{4.n}$, where $n = 1,\ 2,\ 3,\ 4,\ 5,\ 6,\ 7,\ 8$.
	\end{enumerate}
	All these degenerations are listed in section \ref{sec:answer}.
\end{thm}

\begin{rem}
	A posteriori all these smooth threefolds $Y$ satisfy the following conditions
	\begin{enumerate}
		\item $Y$ is rational (see e.g. \cite{IP}),
		\item $\rho(Y)\leq 4$,
		\item $\deg(Y)=(-K_Y)^3\geq 20$,
		\item $b(Y)=h^{1,2}(Y)\leq 3$,
		\item $b(Y)=3$ only if $Y$ is $V_{2.12}$,
		\item $b(Y)=2$ only if $Y$ is $B_4$ or $V_{2.19}$.
	\end{enumerate}
\end{rem}

\section{A proof of the main result.}\label{sec:proof}

\subsection*{A sketch of the proof} Consider a toric Fano threefold $X$ with ordinary double points.
\begin{description}
	\item[(i)] There is only a finite number of such X. All these threefolds X are explicitly classified.
	\item[(ii)] $X$ admits a smoothing --- a Fano threefold $Y$.
	\item[(iii)] Principal invariants of $Y$ can be expressed via invariants of $X$.
	\item[(iv)] Family of smooth Fano threefolds $Y$ is completely determined by its principal invariants.
	\item[(v)] If some smooth Fano threefold $Y$ admits a degeneration to a nodal toric Fano $X$, then the
	pair $(Y,X)$ comes from the steps (i)-(iv).
\end{description}

The following properties of Fano varieties are consequences of Kawamata--Viehweg theorem,
exponential sequence and Leray spectral sequence.

\begin{prop}(See e.g. \cite{IP,KMM})\label{prop: properties of Fanos}
	Let $X$ be an almost Fano with canonical singularities. Then
	\begin{enumerate}
		\item $H^i(X,\OO)=0$ for all i>0,
		\item $\Pic(X)=H^2(X,\ZZ)$,
		\item $\Pic(X)$ is a finitely generated free $\ZZ$-module.
	\end{enumerate}
	If $\pi: Y\to X$ is a resolution of singularities, the listed properties hold also for $Y$ , and $R^\bullet\pi_{*}\OO_Y=\OO_X$ (i.e. canonical singularities are rational).
\end{prop}

Local topology of smoothings is described by the following

\begin{prop}(see e.g. \cite{Cle-d,KK,Voi}) Let $\pi: \XX\to\Delta$ be a smoothing.
	\begin{enumerate}
		\item Restriction $\pi: \XX\backslash\XX_0$ is a locally trivial fibration of smooth topological manifolds, in particular all the smooth fibers are diffeomorphic (this is known as Ehresmann's theorem).
		\item There is a continuous Clemens map $c: \XX\to X_0$ (outside $c^{-1}(\Sing X_0)$) the map $c$ is smooth). Clemens map $c$ is a deformation retraction of $\XX$ to the fiber $X_0$ and respects the radial retraction $\Delta\to 0$. Restriction of $c$ to the smooth fiber $X_t$ is 1-to-1 correspondence outside singular locus of $X_0$.
	\end{enumerate}
	
\end{prop}

These propositions are purely topological, and essentially are the variations of the tubular neighborhood theorem.

\begin{cor}\label{cor: 1}
	$\XX$ and $X_0$ has the same homotopy type (the homotopy equivalences are given by the Clemens map $c : \XX\to X_0$ and the inclusion of the fiber $i_0: X_0\to \XX$). Hence
	\begin{gather*}
		H^2(X_0,\ZZ)=H^2(\XX,\ZZ),\\
		H_2(X_0,\ZZ)=H_2(\XX,\ZZ).
	\end{gather*}
\end{cor}
\begin{cor}\label{cor: 2}
	For $t\ne 0$ all the images $\Image[\{i_t\}_: H_{\bullet}(X_t,\ZZ)\to H_{\bullet}(\XX,\ZZ)]$ coincide.
\end{cor}
\begin{proof}
	Let $U_i$ be the covering of $\Delta\backslash 0$ such that $\pi$ is locally trivial fibration over elements of the covering $U_i$. Consider a pair of points $t,\ s\in U_i$ and a $k$-cycle $\gamma\in H_k(X_t,\ZZ)$. Let $I\subset U$ be an interval between $t$ and $s$ in $U$, and $\gamma_U$ be a $(k+1)$-cycle in $\XX_I$, corresponding to the product of $I$ and $\gamma$ in a fixed trivialization of $\pi$ over $I$. Then the boundary of $\gamma_U$ in $\XX$ is equal to the difference between $\{i_t\}_*\gamma$ and $\{i_s\}_*\gamma$.
\end{proof}

\begin{thm}[\cite{De68}]\label{thm: Hodge numbers are const.}
	Hodge numbers $h^{p,q}(X_t)$ are constant for all $t\in\Delta\backslash 0$.
\end{thm}

\begin{prop}[Semi-continuity theorem, see e.g. \cite{Ha}]\label{thm: semicont} Let $\FF$ be a coherent sheaf on $\XX$, flat over $\OO_\Delta$; put $\FF_t=i_t^*$. Then
	\begin{enumerate}
		\item The Euler characteristic $\chi(X_t,\FF_t)$ does not depend on $t\in\Delta$.
		\item Dimension of $H^i(X_t,\FF_t)$ is upper-semi-continuous as a function of $t$ (i.e. for all $n\in\ZZ$ sets)
		\[
		\{t\in\Delta: h^i(X_t,\FF_t)\geq n\}
		\]
		are closed in Zariski topology).
	\end{enumerate}
\end{prop}

\begin{rem}
	We will use the following trick: if the cohomology of some coherent sheaf $H^i(X_0,\FF)$ vanish, then assume that $\Delta$ is chosen small enough, so that vanishing holds for the cohomology	of all the fibers over $\Delta$.
\end{rem}

\begin{thm}[\cite{Ka97}]\label{thm: gorenstein}
	Let $X_0$ be a variety with canonical singularities, and $\XX$ --- a deformation. Then total space $\XX$ is $\mathbb{Q}$-Gorenstein (Gorenstein if $X_0$ is) and admit only canonical singularities.
\end{thm}

In this case one can use the naive adjunction formula on $X$ (dualizing sheaf coincides with the
canonical one).

Assume that $X_0$ is Gorenstein and admits at most canonical singularities, and either a Calabi-Yau of dimension $\geq 2$ or almost Fano.

\begin{prop}\label{prop: h^i(Xt)=h^0(X0)}
	For all $i$ and $t$
	\[
	h^i(X_t,\OO_{X_t})=h^i(X_0,\OO_{X_0}).
	\]
\end{prop}
\begin{proof}
	Consider $h^i(X_t,\OO_{X_t})_{0<i<\dim X_t}$ as a function of $t$. It is upper-semi-continuous (see \ref{thm: semicont} (ii)),	and equal to 0 for $t=0$ (by the definition if $X$ is Calabi-Yau, or by \ref{prop: properties of Fanos} if $X$ is almost Fano). Hence this function is 0 in some neighborhood of 0. This means it is identical to 0 over $\Delta$ (Theorem \ref{thm: Hodge numbers are const.}). Since
	$h^0(X_t,\OO)=1$ for all $t$, from Proposition \ref{thm: semicont} (i) if follows that $h^n(X_t)=h^n(X_0)$ for all $t$ (it is equal to 0 in case of almost Fano and 1 for Calabi-Yau).
\end{proof}

By Grauerth's theorem $R^i\pi_*\OO=R^i\pi_*\OO(-K_\XX)$, $\dim X_0>i>0$, and $\pi_*\OO(-K_\XX)$ is a locally
free sheaf over $\Delta$ of rank $h^0(X_0,\OO(-K_{X_0}))$. From the degenerations of Leray spectral sequences
$H^i(\Delta, R^j\pi_*\OO(-K_\XX))$ and $H^i(\Delta,R^j\pi_*\OO)$:
	\begin{gather}
		H^i(\XX,\OO_\XX(-K_\XX))=H^i(X_t,\OO_{X_t}(-K_t))=0,\ \dim X_0>i>0,\ t\in\Delta \label{eq:1} \\
		H^i(\XX,\OO_\XX)=H^i(X_t,\OO_{X_t})=0,\ \dim X_0>i>0,\ t\in\Delta \label{eq:2} \\
		H^0(X_t,\OO_{X_t}(-K_{X_t}))=H^0(X_0,\OO_{X_0}(-K_{X_0})),\ t\in\Delta \label{eq:3} \\
		H^0(\XX,\OO_\XX(-K_\XX))=H^0(X_0,\OO_{X_0}(-K_{X_0}))\otimes H^0(\Delta,\OO) \label{eq:4}.
	\end{gather}
By exponential sequence and the vanishing \ref{eq:2} there are isomorphisms
	\begin{gather}
		\Pic(\XX)=H^2(\XX,\ZZ), \label{eq:5} \\
		\Pic(X_t)=H^2(X_t,\ZZ)  \label{eq: 6}
	\end{gather}
Next proposition is a combination of \ref{eq:5} and \ref{cor: 1}:
\begin{prop}\label{prop: pic-isomorpism of X0 and XX}
	$i_0^*: \Pic(\XX)\to\Pic(X_0)$ is an isomorphism.
\end{prop}
\begin{prop}\label{prop: pic XX to X0 inject}
	$i_t^*: \Pic(\XX)\to\Pic(X_0)$ is injective, i.e.
	\begin{equation}
	\Ker i_t^*=0.
	\end{equation}
\end{prop}
\begin{proof}
	Since for all $\gamma\in H_{\bullet}(X_t)$ and $\Gamma\in H^{\bullet}(\XX)$ we have
	\[
	\langle i_t^*(\Gamma),\gamma\rangle=\langle\Gamma,\{i_t\}_*\gamma \rangle,
	\]
	so from non-degeneracy of the coupling on $X_t$ for $t\ne 0$ and Corollary \ref{cor: 2} we conclude that the	spaces $\Ker i_t^*: H^2(\XX,\ZZ)\backslash{\rm tors}\to H^2(X_t,\ZZ)$ coincide for all $t\ne 0$. Isomorphism \ref{eq:5} implies the same holds for $i_t: \Pic(\XX)\to\Pic(X_t)$, i.e. $\Ker i_t=\Ker i_{t'}$ for all $t,t'\in\Delta\backslash 0$.
	
	Consider an element $\LL\in\Ker i_t^*=\cap_{t'\in \Delta\backslash 0}\Ker i_{t'}^*$. Then $\LL$ is invertible sheaf with the property $\LL_{X_t}=\OO_{X_t}$, $t\in\Delta\backslash 0$. If $t\ne 0$ this trivial line bundle has 1-dimensional space of sections:
	\[
	h^0(X_t,\LL_{X_t})=h^0(X_t,\OO_{X_t})=1,
	\]
	so by semi-continuity (Proposition \ref{thm: semicont})
	\[
	h^0(X_0,\LL_{X_0})\geq 1.
	\]
	In the same way
	\[
	h^0(X_0,\LL_{X_0}^{-1})\geq 1.
	\]
	This means $\LL_{X_0}\cong \OO_{X_0}$. So \ref{prop: pic-isomorpism of X0 and XX} implies $\LL\cong\OO_\XX$.
\end{proof}

By \ref{thm: gorenstein} and adjunction formula for all $t\in\Delta$
\begin{equation}\label{eq: K_{Xt}}
-K_{X_t}=-(K_\XX+X_t)|_{X_t}=i_t^*(-K_\XX).
\end{equation}

Consider $\DD\in\Pic(\XX)$. The fibers $X_0=X$ and $X_t$ are algebraically equivalent, so
\begin{equation}\label{eq: DD}
i_0^*(\DD)^{\dim X}=\DD^{\dim X}\cdot X_0=\DD^{\dim X}\cdot X_t=y_t^*(\DD)^{\dim X}.
\end{equation}

\begin{cor}
	Anticanonical degree $(-K_{X_t})^{\dim X_t}$ does not depend on $t\in\Delta$.
\end{cor}

Let $\XX$ be a relative Fano (i.e. $-K_\XX$ is ample over $\Delta$).

\begin{thm}[\cite{Fr}]\label{thm: Friedman} 
	Any Fano threefold $X_0$ with ordinary double points admits a smoothing $\pi: \XX\to\Delta$ with general fiber $X_{t\ne 0}$ being a smooth Fano.
\end{thm}

Friedman's theorem~\ref{thm: Friedman} has a generalization to Gorenstein terminal singularities by Namikawa

\begin{thm}[\cite{Na}]\label{thm: Namikawa}
	Any Gorenstein terminal Fano threefold $X_0$ admits a smoothing $\pi: \XX\to\Delta$ with general fibers $X_{t\ne 0}$ being a smooth Fano.
\end{thm}

\begin{prop}\label{prop: rho X0=rho Xt}
	If $X_0$ is (almost) Fano, then the smoothing is small if and only if two pairs of invariants $(\rho,d)$ coincide ($d$ is defined in \ref{def: d}:
	\begin{gather*}
	\rho(X_0)=\rho(X_t),\\
	d(X_0)=d(X_t).
	\end{gather*}
\end{prop}
\begin{proof}
	Bijectivity of $i_0^*$ and injectivity of $i_t^*$ holds in general context (Proposition \ref{prop: pic XX to X0 inject}). Both groups $\Pic(X_t)$ and $\Pic(X)$ are finitely generated lattices (Proposition \ref{prop: properties of Fanos}). Thus equality $\rho(X_0)=\rho(X_t)$ means	that the morphism $i_t^*(i_0^*)^{-1}$ is an isomorphism of lattice $\Pic(X_0)$ with sub-lattice of finite index in	$\Pic(X_t)$. This index is equal to
	\[
	[\Pic(X_t):\Pic(X_0)]=\left ( \frac{d(X_0)}{d(X_t)}\right )^{1/2}.
	\]
\end{proof}

\begin{thm}[\cite{JR}]\label{thm: JR}
	If $\XX$ is a smoothing, and $X_0$ is a Gorenstein Fano threefold with terminal singularities, then $i_t^*$ is an isomorphism for all $t$.
\end{thm}

\begin{cor}
	Any Gorenstein Fano threefold with terminal singularities admits a smoothing, with	general fiber being a smooth Fano threefold, and all such smoothings are small.
\end{cor}
\begin{proof}
	This is just a union of \ref{thm: Namikawa} and \ref{thm: JR}. 
\end{proof}

\begin{cor}\label{cor: 0}
	Gorenstein Fano threefold $X$ with terminal singularities and its smoothing $Y$ has the same invariants $\rho$, $\deg$, $r$, $d$.
\end{cor}
\begin{proof}
	Equality 3.21 states that $\deg(X)=\deg(Y)$. As a corollary of \ref{thm: JR} we have $\rho(X)=\rho(Y)$. Hence from \ref{eq: DD} and \ref{prop: rho X0=rho Xt} one derives $d(X) = d(Y)$. Finally, \ref{thm: JR} with \ref{eq: K_{Xt}} implies $r(X) = r(Y)$.
\end{proof}

Fano threefold Y has only 2 non-trivial Hodge numbers $h^{1,1}(Y)=h^{2,2}(Y)=\rho(Y)$ and $b(Y)=h^{1,2}(Y)=h^{2,1}(Y)=\frac{1}{2}\rk H^3(Y,\ZZ)$; and some trivial: $h^{0,0}(Y)=h^{3,3}(Y)=1$, all other Hodge numbers are zeroes.

\begin{prop}\label{prop: nodes on nodal 3fold}
	Let $X$ be a nodal threefold, $\widetilde{X}\to X$ --- its small crepant resolution, and $Y$ --- a smoothing of $X$ (in literature transformation from $Y$ to $\widetilde{X}$ is called a conifold transition). Denote the number of nodes on $X$ by $p(X)$. Then
	\begin{equation}\label{eq: nodes on nodal 3fold formula}
	b(Y)=p(X)+b(\widetilde{X})+\rho(Y)-\rho(\widetilde{X}).
	\end{equation}
\end{prop}
\begin{proof}({\it Clemens's argument, see also \cite{NaSteen}}) Compare topological Euler numbers (for non-compact manifolds with a border use Euler number for cohomology with compact support $\chi(M)=\sum_i(-1)^i\dim H_c^i(M,\CC))$ of $\widetilde{X}$ and $Y$\footnote{Alternatively one can compare dimensions of versal deformation spaces for $Y$ and $X$; see also mirror-symmetry explanation \cite{Re,BFKS1}.}.

By throwing away small neighborhoods of all singular points $p_i$ from $X$, we construct a manifold	with the border $M$. Punctured neighborhood of ordinary double point on $X$ is isomorphic to tangent bundle on real sphere $TS^3$ without the 0-section: if $\sum_{i=1}^{4}z_i^2=0$, $z=x+yi$ then $x$ and $y$ can be considered as a pair of nonzero orthogonal (with respect to a standard Euclidean metric) vectors in $\RR^4$ of the same length $r$; vector $x/r$ is a point in $(n-1)$-dimensional sphere of radius i$1$, and $y$ is a tangent vector in that point. This shows that a neighborhood of ordinary double point on $X$ is	isomorphic to $S^2\times S^3$. After crepant resolution it is patched by $S^2\times D^4$, and after smoothing --- by $D^3\times S^3$. Hence
\begin{gather*}
\chi(\widetilde{X})=\chi(M)+p\cdot\chi(S^2),\\
\chi(Y)=\chi(M)+p\cdot\chi(S^3).
\end{gather*}
This implies
\[
\chi(\widetilde{X})=\chi(Y)+2p.
\]
But
\begin{gather*}
\chi(Y)=2+2\rho(Y)-2b(Y),\\
\chi(\widetilde{X})=2+2\rho(\widetilde{X})-2b(\widetilde{X}).
\end{gather*}	
\end{proof}

\begin{prop}\label{prop: b p rho formula}
	If $X$ is a nodal toric threefold corresponding to a polytope with $v$ vertices, $p$ quadrangular faces (i.e. nodes) and $f-p$ triangular faces (smooth fixed points), then $H^3(\widetilde{X})=0$, $\rho(\widetilde{X})=v- 3$. So for smoothing $Y$ of $X$, there is a relation
	\[
	b(Y)=p+\rho(X)-(v-3).
	\]
\end{prop}
\begin{proof}
	Since $\widetilde{X}$ is nonsingular, $\Pic(\widetilde{X})$ and $\Cl(\widetilde{X})$ coincides. But the resolution $\widetilde{X}\to X$ is small, hence the proper transform is the bijection between Weyl divisors on $\widetilde{X}$ and $X$, i.e. $\Cl(\widetilde{X})=\Cl(X)$. This implies $\rho(\widetilde{X})=\rk\Pic(\widetilde{X})=\rk\Cl(X)=v-3$. Therefore proposition \ref{prop: nodes on nodal 3fold} in our case is equivalent to the equality \ref{eq: nodes on nodal 3fold formula}.
\end{proof}

\begin{thm}[\cite{MM2,MM3}]\label{thm: same invariants - same class}
	Two smooth Fano threefolds $Y_1$, $Y_2$ with coincident sets of principal invariants $\rho$, $r$, $\deg$, $b$, $d$ lie in one deformation class. There are only 105 such classes\footnote{In the first version of \cite{MM2} one family $V_{4.13}$ was missing, it was corrected in 2003.}. They are	explicitly listed in \cite{MM2}, and nonempty.
\end{thm}

Let us say that smooth Fano threefold $Y$ is determined by its invariants ($\rho,\ r,\ \deg,\ b$), if for any smooth Fano threefold $Y'$ equalities $\rho(Y')=\rho(Y),\ \rho(Y')=\rho(Y),\ \deg(Y')=\deg(Y),\ b(Y')=b(Y)$ imply that $Y$ and $Y'$ lie in one deformation class. According to \cite{MM3}, only 19 of 105 families of smooth Fano threefolds are not determined by invariants $\rho,\ r,\ \deg,\ b$.

\begin{lem}\label{lem: existing of sm. for nodal Fanos}
	For any nodal Fano threefold $X$ there exists only one (up to deformations) smooth Fano $Y$, such that $Y$ is a smoothing of $X$.
\end{lem}
\begin{proof}
	$X$ has a smoothing --- a smooth Fano variety $Y$ (see \ref{thm: Friedman}). Principal invariants of $Y$ (see \ref{def: principal invariants}) are explicitly computable from invariants of $X$ (see \ref{cor: 0}, \ref{prop: b p rho formula}). Deformation class of $Y$ is uniquely determined by its principal invariants \ref{thm: same invariants - same class}.
\end{proof}

\begin{cor}\label{cor: degeneration inv crit}
	Suppose $Y$ is determined by ($\rho$, $r$, $\deg$, $b$). Then nodal Fano threefold $X$ is a degeneration of $Y$ if and only if $\rho(X)=\rho,\ r(X)=r,\ \deg(X)=\deg,\ b(X)=b$. If $Y$ is not determined by $(\rho$, $r$, $\deg$, $b$), then $X$ is a degeneration of $Y$ if and only if $\rho(X)=\rho,\ r(X)=r,\ \deg(X)=\deg,\ b(X)=b,\ d(X)=d$.
\end{cor}

The proof of lemma \ref{lem: existing of sm. for nodal Fanos} works in higher generality --- not only in case of toric varieties, but for any nodal Fano threefolds (and also it is easy to generalize it to the case of Fano threefolds with Gorenstein terminal singularities). In the next part of the paper we restrict ourselves to the case of toric varieties $X$\footnote{For simplicity of computations, and applications (see \ref{sec:applications})}.

There is an effective algorithm describing all the reflexive polytopes (i.e. Gorenstein toric Fano varieties) in any fixed dimension \cite{KS}. Number of such polytopes grows fast enough: there are 16 polygons, 4319 polytopes in 3-dimensional space, and 473800776 4-dimensional polytopes.

We are interested in the particular case of nodal toric Fano threefolds.
We used PALP software package \cite{PALP,KS} to form a list of such varieties.
There are 100 of them, 18 are smooth and are not deformations of other Fano manifolds (theorem \ref{thm: same invariants - same class}).
For non-smooth cases Picard number is at most $4$.
All these varieties are listed in the table of Section~\ref{sec:answer}\footnote{
An explicit description of all nodal toric Fano threefolds is given in \cite{Ni}.
All terminal toric Fano threefolds are classified in\cite{Ka}.
All Gorenstein toric Fano threefolds are classified in \cite{KS}.
}.

So let us compute invariants of the smoothing $Y$ of toric nodal Fano $X$.

Let $\pi: \widetilde{X}\to X$ be some small crepant resolution of $X$, and $p(X)$ be a number of nodes on $X$.

\begin{proof}[Proof of \ref{thm: admitting toric degeneration}] Assume smooth Fano threefold $Y$ is degenerated to $X$.	As shown in \ref{cor: 2}, varieties $X$ and $Y$ have the same Picard number, index, anticanonical degree	and invariant $d$. Denote them by
\begin{gather*}
\rho(X)=\rho(Y)=\rho,\\
r(X)=r(Y)=r,\\
(-K_X)^3=(-K_Y)^3=\deg,\\
d(X)=d(Y)=d.
\end{gather*}
Since $\widetilde{X}$ is toric, all its odd cohomology vanish: $H^3(\widetilde{X},\mathbb{Q})=0$. This implies (see \ref{prop: nodes on nodal 3fold}, \ref{eq: nodes on nodal 3fold formula}, \ref{prop: b p rho formula}):
\[
b(Y)=p(X)+\rho(X)-\rho(\widetilde{X}).
\]
Put $b=p(X)+\rho(X)-\rho(\widetilde{X})$.

What is left to do is to compute invariants $\rho,\ r,\ \deg, b,\ d$ of $X$ (this is done in section 4), and pick up a unique family of smooth varieties $Y$ with invariants $\rho(Y)=\rho,\ r(Y)=r,\ \deg(Y)=\deg,\ b(Y)=b,\ d(Y)=f$, in the table of \cite{MM1}.
\end{proof}

The remaining statements in this chapter serve to simplify the computations. Picard number of nodal toric Fano threefold is either 1, 2, 3 or 4 (see \cite{Ni} and table in \ref{sec:answer}). Hence smooth non-toric Fano varieties $\rho\geq 5$ (i.e. non-toric variety of degree 28 with $\rho=5$ and products $\PP^1\times S_{d=11-\rho}$ of the line $\PP^1$ with del Pezzo surface $S_d$ of degree $d\leq 5$) has no small toric deformations.

In 55 of 82 cases of singular $X$ the smoothing $Y$ is determined by its invariants $(\rho,\ b,\ r,\ \deg)(Y)=(\rho,\ b,\ r,\ \deg)(X)$. In these cases the routine computation of invariant $d(X)$ may be omitted.

There are eight exceptional sets of invariants $(\rho,\ b,\ r,\ \deg)$ corresponding to 17 families of Fano varieties
listed in the following table:

\begin{center}
	\begin{longtable}{|c|c|c|c|}
		\caption{}\label{table 1}\\
		\hline
		$\rho$ & $\deg$ & $b,\ r$ & \text{smooth} $Y$ \\
		\hline
		\endfirsthead
		\multicolumn{4}{c}%
		{\tablename\ \thetable\ \textit{}} \\
		\hline
		$\rho$ & $\deg$ & $b,\ r$ & \text{Smooth} $Y$ \\
		\hline
		\endhead
		\hline \multicolumn{4}{r}{\textit{}} \\
		\endfoot
		\hline
		\endlastfoot
		2 & 30 & 0,1 & $V_{2.22}[-24],V_{2.24}[-21]$ \\ 
		2 & 46 & 0,1 & $V_{2.30}[-12],V_{2.31}[-13]$ \\ 
		3 & 36 & 0,1 & $V_{3.17}[28],V_{3.18}[26]$ \\ 
		3 & 38 & 0,1 & $V_{3.19}[24],V_{3.20}[28],V_{3.21}[22]$ \\ 
		3 & 42 & 0,1 & $V_{3.23}[20],V_{3.24}[22]$ \\ 
		4 & 32 & 0,1 & $V_{4.4}[-40],V_{4.5}[-39]$ \\ 
		2 & 54 & 0,2 & $V_{2.33},V_{2.34}$ \\ 
		3 & 48 & 0,2 & $V_{3.27},V_{3.28}$ \\ 
	\end{longtable}
\end{center}

\begin{rem}
	Smooth varieties $V_{2.33},\ V_{2.34},\ V_{3.27},\ V_{3.28}$ are toric.
\end{rem}
\begin{rem}
	In table \ref{table 1} the number in brackets after smooth Fano $Y$ is its invariant $d(Y)$ (see \cite[Proposition 7.35]{MM3}).
\end{rem}

\section{Computation of discriminants}\label{sec:computation}

\begin{thm}[see e.g. \cite{Ba93}]
Let $X$ be nonsingular and proper (probably not projective) toric variety.
Cohomology ring $H^\bullet(X,\mathbb{Q})$ is generated by classes of invariant divisors $D_{\rho_i}$.
The relations in this ring are generated by the so-called Stanley--Reisner relations:
for all $J\subset\Sigma^{(1)}$, not contained in any face $\Delta$ one has
	\[
	\prod_{j\in J\subset\Sigma^{(1)}}D_{\rho_j}=0,
	\]
	and relations implied by the triviality of principal divisors, i.e. for all $m\in M$
	\[
	\sum_i\langle m,\rho_i\rangle D_{\rho_i}=0.
	\]
\end{thm}

This means that in the cohomology ring of a smooth toric variety all the relations are generated by naive ones: intersection of $k$ different divisors is empty if the corresponding 1-dimensional faces are not contained in one $k$-dimensional face $\sigma$. If they are contained, then the corresponding divisors intersect transversely in $(d-k)$-dimensional orbit corresponding to the face $\sigma$.

\begin{lem}\label{lem: equations}
	Let $X_\Sigma$ be a smooth toric $n$-fold. Consider a homogeneous system of linear equations
	\begin{gather*}
	x_{j_1\ldots j_n}=0,\ \text{if}\ \{\rho_{j_1},\ldots, \rho_{j_n}\}\ \text{is not a cone in}\ \Sigma,\\
	\sum\langle m,\rho_j\rangle x_{j_1\ldots j_{i-1}jj_{i+1}\ldots j_n}=0
	\end{gather*}
	This system has a unique solution up to rescaling. Choose a unique solution that satisfy $x_{j_1\ldots j_n}=1$,
	if $\{\rho_{j_1}\ldots\rho_{j_n}\}$ is a cone in $\Sigma$. Then the numbers $x_{j_1\ldots j_n}$ are equal to the intersection numbers of	divisors $D_{j_1}\cdot\ldots\cdot D_{j_n}$ on $X_\Sigma$.
\end{lem}
\begin{prop}\label{prop: equation}
	For Weyl divisor $\sum a_\rho D_\rho$ the condition of local principality in ordinary double point on toric threefold is the following --- sum of coefficients at invariant irreducible divisors corresponding to the vertices of the diagonal $\rho_A\rho_C$ of quadrangle $\rho_A\rho_B\rho_C\rho_D$ is equal to the sum at
	the vertices of $\rho_B\rho_D$:
	\[
	a_{\rho_A}+a_{\rho_C}=a_{\rho_B}+a_{\rho_D}.
	\]
\end{prop}
\begin{lem}\label{lem: exact seq for Pic}
	Let $X$ be a nodal toric Fano threefold. Then $\Pic(X)$ is determined from the exact sequence
	\[
	0\longrightarrow\Pic(X)\longrightarrow\Pic(\widetilde{X})\overset{\phi}{\longrightarrow}\oplus_{ABCD}\ZZ,
	\]
	where the sum is taken over all basic quadrangles $\rho_A\rho_B\rho_C\rho_D$ for $X$, $\phi=\oplus_{ABCD}\phi_{ABCD}$, and $\phi_{ABCD}(\sum a_\rho D_\rho)=(a_{\rho_A}-a_{\rho_B}+a_{\rho_C}-a_{\rho_D})$.
\end{lem}

\begin{rem}\label{rem: computation}
	By virtue of lemmas \ref{lem: equations} and \ref{lem: exact seq for Pic}, one may effectively compute the intersection theory	on $\Pic(X)$ for $\mathbb{Q}$-Gorenstein toric $X$, admitting a small resolution $f: \widetilde{X}\to X$ (e.g. all nodal threefolds $X$ satisfy this property). Self-intersection $D^n$ of Cartier divisor $D\in\Pic(X)$ is equal to intersection of its pullback $\widetilde{D}=f^*D$ to $\widetilde{X}$. Class group of Weyl divisors is invariant modulo small resolutions, divisor $\widetilde{D}$ is represented by the same Weyl divisor as $D$ (by	the pullback).
	
	Therefore to find the intersections on $\Pic(X)$ one need to solve two systems of linear equations:	one on intersection numbers $D_{i_1}\cdot\ldots\cdot D_{i_n}$ described in \ref{lem: equations}, and another one --- the equations \ref{prop: equation} cutting $\Pic(X)$ as a subgroup of $\Pic(\widetilde{X})$\footnote{The PARI/GP script realizing this algorithm is available at \url{http://www.mi.ras.ru/~galkin/work/NodalToric3foldPicard.gp}.}.
\end{rem}

\subsection*{Notation} Let $M$ be a integer matrix of size $3\times v$. Denote by $\Delta(M)$ the convex hull of columns of $M$. Assume $M$ is chosen in such a way that 0 is contained in the interior of $\Delta(M)$, and none of $M$'s columns lie in the convex hull of the others. By $\PP(M)$ denote the toric Fano variety corresponding to the polytope $\Delta(M)$. Let $D_i$ be invariant Weyl divisor corresponding to $i$th vertice of $\Delta(M)$, and $G_1,\ldots G_\rho$ be the generators of $\Pic(\PP(M))$.

In order to compute $d$, we find first all the intersection numbers of the elements in the base of $\Pic(\PP(M))$, and then compute the discriminant. We use \ref{rem: computation} for the computation of intersection numbers of divisors in $\Pic(\PP(M))$ --- compute the ring $H^{\bullet}(\widetilde{\PP}(M))$, intersections in Picard group $\Pic(\widetilde{\PP}(M))$ of small crepant resolution\footnote{We choose arbitrary maximal crepant resolution as explained in \ref{rem: computation}, the answer does not depend on the projectivity of the resolutions.} $\phi: \widetilde{\PP}(M)\to\PP(M)$, and then intersections in $\PP(M)$ is just the restriction from $\widetilde{\PP}(M)$.

As an example we produce this computation for case with invariants $(\rho=2,\ \deg=30,\ b=0)$\footnote{All the other cases are available at \url{http://www.mi.ras.ru/~galkin/work/NodalToric3foldPicard.pdf}.}.
\newline
\newline
{\bf Case 4.6} ($v=9,\ f=10$). 
\[
M=
\begin{pmatrix}
	1 & 0 & 0 & 0 & 0 & -1 & 1 & 1 & -1 \\
	0 & 1 & 0 & 0 & -1 & 0 & 1 & 0 & -1 \\
	0 & 0 & 1 & -1 & 0 & 0 & 0 & -1 & 1 
\end{pmatrix}
\]
$G_1=D_1+D_4+D_5+D_8,\ G_2=-D_1+D_6+D_9$.
\begin{gather*}
{\rm int}(aG_1+bG_2,\ aG_1+bG_2,\ aG_1+bG_2)=(aG_1+bG_2)^3=a^3+6ba^2-2b^3\\
-K=G_1+2G_2\\
d=-24
\end{gather*}
\newline
{\bf Case 4.7} ($v=10,\ f=11$). 
\[
M=
\begin{pmatrix}
1 & 0 & 0 & 0 & 0 & -1 & 0 & -1 & -1 & -1\\
0 & 1 & 0 & 0 & -1 & 0 & 1 & 1 & 0 &-1 \\
0 & 0 & 1 & -1 & 0 & 0 & -1 & 0 & 1 & 1
\end{pmatrix}
\]
$G_1=D_7+D_8+D_9,\ G_2=D_2+D_3+D_5-D_6+D_{10}$.
\begin{gather*}
(aG_1+bG_2)^3=-2a^3+6ba^2-3b^3\\
-K=3G_1+2G_2\\
d=-24
\end{gather*}
\newline
{\bf Case 4.8} ($v=9,\ f=10$). 
\[
M=
\begin{pmatrix}
1 & 0 & 0 & -1 & 0 & -1 & 1 & 0 & 1\\
0 & 1 & -1 & 0 & 0 & 1 & 1 & -1 & 0 \\
0 & 0 & 0 & 0 & 1 & 0 & -1 & 1 & -1
\end{pmatrix}
\]
$G_1=-D_1+2D_3+D_4-D_7+D_8,\ G_2=D_1+D_7+D_9$.
\begin{gather*}
(aG_1+bG_2)^3=3ba^2+6b^2a\\
-K=G_1+G_2\\
d=-21
\end{gather*}

\section{The description of toric degenerations of smooth Fano threefolds}\label{sec:answer}

As mentioned in corollary \ref{cor: degeneration inv crit}, for the determination of all the possible types of toric degenerations of Fano threefolds $Y$ , we need to compute the invariants $\rho(X)$, $r(X)$, $\deg(X)$, $b(X)$ (and
sometimes $d(X)$) of all nodal toric Fano threefolds. For these computations we used the program\footnote{\url{http://www.mi.ras.ru/~galkin/work/NodalToric3foldPicard.gp}} based on algorithm described in \ref{lem: equations}, \ref{lem: exact seq for Pic}, \ref{rem: computation}. The results of these computations are exposed in the table \ref{table 2}.

First 4 columns list Fano threefolds $Y$ and its invariants computed in \cite{MM3}.

In 5th column we list the value of invariant $d(Y)$ for cases when $Y$ is not determined by $(\rho,\ r,\ \deg,\ b)$.

In 6th column we list main combinatorial invariants of toric $X$ (degeneration of $Y$) --- number of vertices, nodes and torus-fixed points.

In 7th column we list the number of toric degenerations $X$ of smooth $Y$ with invariants listed in 6th column.

\begin{rem}
	There is a linear relation \ref{prop: b p rho formula} between $\rho,\ b,\ v,\ p$:
	\[
	v-p=3+\rho-b.
	\]
\end{rem}
\begin{rem}
	Varieties $V_{2.34},\ V_{3.25},\ V_{3.26},\ V_{3.28},\ V_{4.9}$ are smooth toric varieties that admit degenerations
	to singular nodal toric varieties. The rest of smooth varieties listed in the table are non-toric.
\end{rem}
\begin{rem}
	Fano variety $V_{4.13}$ (of degree 26)\footnote{This threefold is missing in the original version of \cite{MM1}.} does not admit small toric degenerations.
\end{rem}

\begin{center}
	\begin{longtable}{|c|c|c|c||c|c|c|}
		\caption{}\label{table 2}\\
		\hline
		$V_{22}$ & 1 & 22 & 0 &  & (13,9,13) & 1 \\ \hline
		$B_4$ & 1 & 32 & 2 &  & (8,6,6) & 1 \\ \hline
		$B_5$ & 1 & 40 & 0 &  & (7,3,7) & 1 \\ \hline
		$Y$ & $\rho$ & $\deg$ & $b$ & $[d]$ & $(v,p,f)(X)$ & $\#(X)$\\
		\hline
		\endfirsthead
		\multicolumn{7}{c}%
		{\tablename\ \thetable\ \textit{}} \\
		\hline
		$Y$ & $\rho$ & $\deg$ & $b$ & $[d]$ & $(v,p,f)(X)$ & $\#(X)$ \\
		\hline
		\endhead
		\hline \multicolumn{7}{r}{\textit{}} \\
		\endfoot
		\hline
		\endlastfoot
		$Q$ & $1$ & $54$ & $0$ &  & $(5,1,5)$ & $1$ \\ \hhline{=|=|=|=|=|=|=}
		$V_{2.12}$ & 2 & 20 & 3 &  & (14,12,12) & 1 \\ \hline
		$V_{2.17}$ & 2 & 24 & 1 &  & (12,8,12) & 1 \\ \hline
		$V_{2.19}$ & 2 & 26 & 2 &  & (11,8,10) & 1 \\ \hline
		$V_{2.20}$ & 2 & 26 & 0 &  & (11,6,12) & 2 \\ \hline
		$V_{2.21}$ & 2 & 28 & 0 &  & (10,5,11) & 2 \\ \hline
		$V_{2.21}$ & 2 & 28 & 0 &  & (11,6,12) & 1 \\ \hline
		$V_{2.23}$ & 2 & 30 & 1 &  & (9,5,9) & 1 \\ \hline
		$V_{2.22}$ & 2 & 30 & 0 &  & (10,5,11) & 1 \\ \hline
		$V_{2.22}$ & 2 & 30 & 0 &$[-24]$& (9,4,10) & 1 \\ \hline
		$V_{2.24}$ & 2 & 30 & 0 &$[-21]$& (9,4,10) & 1 \\ \hline
		$V_{2.25}$ & 2 & 32 & 1 &  & (8,4,8) & 1 \\ \hline
		$V_{2.25}$ & 2 & 32 & 1 &  & (9,5,9) & 1 \\ \hline
		$V_{2.26}$ & 2 & 34 & 0 &  & (10,5,11) & 1 \\ \hline
		$V_{2.26}$ & 2 & 34 & 0 &  & (8,3,9) & 1 \\ \hline
		$V_{2.26}$ & 2 & 34 & 0 &  & (9,4,10) & 1 \\ \hline
		$V_{2.27}$ & 2 & 38 & 0 &  & (7,2,8) & 1 \\ \hline
		$V_{2.27}$ & 2 & 38 & 0 &  & (8,3,9) & 2 \\ \hline
		$V_{2.28}$ & 2 & 40 & 1 &  & (7,3,7) & 1 \\ \hline
		$V_{2.29}$ & 2 & 40 & 0 &  & (7,2,8) & 1 \\ \hline
		$V_{2.29}$ & 2 & 40 & 0 &  & (8,3,9) & 1 \\ \hline
		$V_{2.30}$ & 2 & 46 & 0 &$[-12]$& (6,1,7) & 1 \\ \hline
		$V_{2.31}$ & 2 & 46 & 0 &$[-13]$& (6,1,7) & 1 \\ \hline
		$V_{2.31}$ & 2 & 46 & 0 &$[-13]$& (7,2,8) & 1 \\ \hline
		$V_{2.32}$ & 2 & 48 & 0 &  & (6,1,7) & 1 \\ \hline
		$V_{2.34}$ & 2 & 54 & 0 &  & (6,1,7) & 1 \\ \hhline{=|=|=|=|=|=|=}
		
		$V_{3.7}$  & 3 & 24 & 1 &  & (12,7,13) &  1 \\ \hline
		$V_{3.10}$ & 3 & 26 & 0 &  & (11,5,13) &  1 \\ \hline
		$V_{3.11}$ & 3 & 28 & 1 &  & (10,5,11) &  1 \\ \hline
		$V_{3.12}$ & 3 & 28 & 0 &  & (10,4,12) &  1 \\ \hline
		$V_{3.12}$ & 3 & 28 & 0 &  & (11,5,13) &  1 \\ \hline
		$V_{3.13}$ & 3 & 30 & 0 &  & (10,4,12) &  2 \\ \hline
		$V_{3.13}$ & 3 & 30 & 0 &  & (9,3,11) &  1 \\ \hline
		$V_{3.14}$ & 3 & 32 & 1 &  & (8,3,9) &  1 \\ \hline
		$V_{3.15}$ & 3 & 32 & 0 &  & (10,4,12) &  1 \\ \hline
		$V_{3.15}$ & 3 & 32 & 0 &  & (9,3,11) &  3 \\ \hline
		$V_{3.16}$ & 3 & 34 & 0 &  & (8,2,10) &  1 \\ \hline
		$V_{3.16}$ & 3 & 34 & 0 &  & (9,3,11) &  1 \\ \hline
		$V_{3.17}$ & 3 & 36 & 0 &$[28]$& (8,2,10) &  2 \\ \hline
		$V_{3.17}$ & 3 & 36 & 0 &$[28]$& (9,3,11) &  1 \\ \hline
		$V_{3.18}$ & 3 & 36 & 0 &$[26]$& (8,2,10) &  1 \\ \hline
		$V_{3.18}$ & 3 & 36 & 0 &$[26]$& (9,3,11) &  1 \\ \hline
		$V_{3.19}$ & 3 & 38 & 0 &$[24]$& (7,1,9) &  1 \\ \hline
		$V_{3.19}$ & 3 & 38 & 0 &$[24]$& (8,2,10) &  1 \\ \hline
		$V_{3.20}$ & 3 & 38 & 0 &$[28]$& (7,1,9) &  1 \\ \hline
		$V_{3.20}$ & 3 & 38 & 0 &$[28]$& (8,2,10) &  1 \\ \hline
		$V_{3.20}$ & 3 & 38 & 0 &$[28]$& (9,3,11) &  1 \\ \hline
		$V_{3.21}$ & 3 & 38 & 0 &$[22]$& (8,2,10) &  1 \\ \hline
		$V_{3.22}$ & 3 & 40 & 0 &  & (7,1,9) &  1 \\ \hline
		$V_{3.23}$ & 3 & 42 & 0 &$[20]$& (7,1,9) &  1 \\ \hline
		$V_{3.23}$ & 3 & 42 & 0 &$[20]$& (8,2,10) &  1 \\ \hline
		$V_{3.24}$ & 3 & 42 & 0 &$[22]$& (7,1,9) &  1 \\ \hline
		$V_{3.24}$ & 3 & 42 & 0 &$[22]$& (8,2,10) &  1 \\ \hline
		$V_{3.25}$ & 3 & 44 & 0 &  & (7,1,9) &  1 \\ \hline
		$V_{3.26}$ & 3 & 46 & 0 &  & (7,1,9) &  1 \\ \hline
		$V_{3.28}$ & 3 & 48 & 0 &  & (7,1,9) &  1 \\ \hhline{=|=|=|=|=|=|=}
		$V_{4.1}$  & 4 & 24 & 1 &  & (12,6,14) & 1 \\ \hline
		$V_{4.2}$  & 4 & 28 & 1 &  & (10,4,12) & 1 \\ \hline
		$V_{4.3}$  & 4 & 30 & 0 &  & (10,3,13) & 1 \\ \hline
		$V_{4.4}$  & 4 & 32 & 0 &$[-40]$& (9,2,12) & 1 \\ \hline
		$V_{4.5}$  & 4 & 32 & 0 &$[-39]$& (9,2,12) & 1 \\ \hline
		$V_{4.6}$  & 4 & 34 & 0 &  & (10,3,13) & 1 \\ \hline
		$V_{4.6}$  & 4 & 34 & 0 &  & (9,2,12) & 1 \\ \hline
		$V_{4.7}$  & 4 & 36 & 0 &  & (8,1,11) & 2 \\ \hline
		$V_{4.7}$  & 4 & 36 & 0 &  & (9,2,12) & 1 \\ \hline
		$V_{4.8}$  & 4 & 38 & 0 &  & (8,1,11) & 1 \\ \hline
		$V_{4.9}$  & 4 & 40 & 0 &  & (8,1,11) & 1 \\ \hline		
	\end{longtable}
\end{center}

Any smooth Fano threefold not listed in the table does not admit any small toric degenerations, since none of nodal toric Fano threefolds has the proper invariants.

\section{Some applications}\label{sec:applications}

We used the classification of smooth Fano threefolds to compare numerical invariants of $X$ and $Y$.
We want to point out that the classification was required only to identify the ``names'' of Fano manifolds.
Our considerations also proved the existence of smooth threefolds such that their sections are canonically embedded curves,
and with principal invariants written in Section~\ref{sec:answer}.
In particular we have provided a new proof that there exists $V_{22}$,
a smooth threefold with $b_2=1$ such that its linear sections
are canonically embedded curves of genus $12$,
a threefold that Fano missed.

Our constructions also help to compute some other invariants of Fano threefolds.

\begin{prop}\label{c1}
Let $X$ be a Gorenstein toric Fano variety with isolated singularities.
Then there exists a smooth anticanonical section $S\in|-K_X|$, and it is a Calabi--Yau variety.
\end{prop}
\begin{proof}
	It is a simple corollary of Bertini theorem.
\end{proof}

\begin{prop}
	Smoothings $X_t$ of Gorenstein Calabi--Yau $X_0$ are Calabi--Yau varieties.
\end{prop}
\begin{proof}
	\ref{prop: h^i(Xt)=h^0(X0)} implies $h^i(X_t,\OO)=0$ for $0<i<\dim X_t$. Hence by \ref{prop: pic-isomorpism of X0 and XX} and \ref{eq: K_{Xt}} we have the trivializations $K_{\XX}|_{X_0}=K_{X_0}=\OO_{X_0}\Rightarrow K_{\XX}=\OO_{\XX}$, so $K_{X_t}=K_{\XX}|_{X_t}=\OO_{\XX}|_{X_t}=\OO_{X_t}$.
\end{proof}

\begin{cor}\label{c2}
	Anticanonical sections $X_t$ are the deformations of anticanonical sections of $X_0$.
\end{cor}
\begin{proof}
	If $Y_0$ is some anticanonical section corresponding to the element $y_0\in H^0(X_0,-K_{X_0})$, then $Y$ is anticanonical section of $X$ corresponding to $y_0\otimes 1\in H^0(X_0,-K_{X_0})\otimes H^0(\Delta,\OO_\Delta)=H^0(\XX,-K_{\XX})$	(see \ref{eq:4}), and establishing the required deformation. From the exact sequence
	\[
	0\longrightarrow\OO_{\XX}((-m-1)K_{\XX})\longrightarrow\OO(-mK_{\XX})\longrightarrow\OO_{\mathcal{Y}}(-mK_{\XX})\longrightarrow 0,
	\]
	vanishings \ref{prop: properties of Fanos} and \ref{eq:1}, \ref{eq:2}, \ref{eq:3}, \ref{eq:4} (with the similar for $\OO(-mK_{\XX})$) we deduce that Hilbert polynomial $\mathcal{Y}_t$ does not depend on $t$, so the family $\mathcal{Y}_t$ is flat.
\end{proof}

\begin{cor}\label{c3}
	If there exists a smooth anticanonical section of $X_0$, then general anticanonical section of $X_t$ for general $t$ is smooth.
\end{cor}
\begin{cor}\label{c4}
	If smooth Fano $Y$ is a smoothing of Gorenstein toric Fano variety with isolated singularities then there exists a smooth anticanonical section $S'\in|-K_Y|$.
\end{cor}
\begin{proof}
	This is a corollary of \ref{c1}, \ref{c2} and \ref{c3}.
\end{proof}
For a subvariety $Z\subset X$ (and divisor $H$) denote by $I^X_H$ the fundamental term of $I$-series of $X$ (with respect to $H$), and by $I^{X\to Z}$ --- the fundamental term of $I$-series of $Z$ restricted from $X$ (see \cite{Gi97,Ga00,Pr1}). Givental's theorem~\cite{Gi97} compute the $I$-series of smooth complete intersection $Z$ of sections of numerically effective line bundles $\OO(Z_i)$, when $Z$ if an almost Fano inside smooth toric $X$ (the similar statement holds for any smooth complete intersection in singular toric variety as well, see \cite{Pr3}).
In particular the $I$-series of toric Fano $X=\PP(\Delta)$ of index $r(Y)>1$ is equal to the
series of constant terms $\pi_f$ of Laurent polynomial $f(x)=\sum_{m\in\Delta\cap M}x^m-1$. Let $[1]g$ denote the
coefficient at $1=x_0$ in Laurent series $f=g(x)$. Then $\pi_f(t)=[1]e^{tf(x)}$.

Let $X$ be a small toric degeneration of $Y$ and $\phi:\widetilde{X}\to X$ be some small crepant resolution.

\begin{prop}\label{prop: I-series}
	$I$-series for $I^{Y\to S'}$ restricted from $\Pic(Y)$ to $S'$ is equal to $I$-series for $\widetilde{X}$ restricted
	from $\Image[\Pic(X)\to\Pic(\widetilde{X})]$ to $\phi^{-1}(S)\cong S$.
\end{prop}
\begin{proof}
	By \ref{c1} the general element $S$ of anticanonical linear system $|-K_X|$ of Gorenstein Fano $X$ with isolated terminal singularities is a smooth Calabi--Yau. As we have shown in \ref{c2}, smooth anticanonical sections of $X$ and its smoothing $Y$ lie in the same deformation class. Picard group	$\Pic X$ is isomorphic to $\Pic Y$ by the assumption of smallness. Consider $\mathcal{H}\in\Pic(X)$. Then 
	\[
	I_{\mathcal{H}_S}^{\widetilde{X}\to S}=I_{\mathcal{H}_{S'}^{Y\to S'}}.  \]
\end{proof}

\begin{ex}
Consider Laurent polynomial
\[
	f_1=xyz+x+y+z+x^{-1}+y^{-1}+z^{-1},
\]
its Newton polytope $\Delta=\Delta(f)$,
and the corresponding toric variety $X=\PP(\Delta^\vee)$. 
One can construct variety $X$ as follows:
let $\hat{X}$ be a blowup of a point on $\PP^1\times\PP^1\times\PP^1$;
then $\hat{X}$ is almost Fano, but not Fano since the proper transforms of coordinate lines do not intersect $-K_{\hat{X}}$; 
the contraction $X$ of these lines is a Fano variety with 3 nodes --- images of contracted curves, 
and $\hat{X}$ is its small crepant resolution.
Three ordinary double points of $X$ correspond to 3 quadrangular faces $(xyz,x,y,z^{-1})$, $(xyz, x, z, y^{-1})$ and $(xyz, y, z, x^{-1})$.
Since $\hat{X} \to X$ is a small crepant resolution, we have $\deg(X) = (-K_X)^3 = (-K_{\hat{X}})^3 = (-K_{\PP^1\times\PP^1\times\PP^1})^3 - 8 = 40$.
Let $Y$ be a Fano smoothing of $X$.
Consider a Laurent polynomial with Newton polytope $\Delta$:
\[
	f_a=\sum a_mx^m=a_{xyz}xyz+a_xx+a_yy+a_zz+a_{x^{-1}}x^{-1}+a_{y^{-1}}y^{-1}+a_{z^{-1}}z^{-1}.
\]
It corresponds to the divisor $\sum b_mD_m\in\Pic(\hat{X})\otimes\CC$, such that $a_m=\exp{2\pi i b_m}$.
This divisor is a pullback of Cartier divisor on $X$, if its coefficients satisfy 3 conditions of local principality
\begin{gather*}
	b_{xyz}+b_{x^{-1}}=b_y+b_z,\\
	b_{xyz}+b_{y^{-1}}=b_x+b_y,\\
	b_{xyz}+b_{z^{-1}}=b_x+b_z.
\end{gather*}
Principal divisors are
\[
	(b_x+b_y+b_z) xyz + b_x x + b_y y + b_z z - b_x x^{-1} - b_y y^{-1} - b_z z^{-1}.
\]
	$X$ has index 2, and its Picard group is generated by $-D_{xyz}+D_{x^{-1}}+D_{y^{-1}}+D_{z^{-1}}$.
Modulo principal divisors the Laurent polynomial corresponding to $\alpha$-multiple of a generator of $\Pic(X)$
is equal to $f_t = t(xyz+x+y+z+x^{-1}+y^{-1}+z^{-1})$, $t=\exp{\pi i\alpha}$.
	
	By the virtue of \cite{Gi97} $I^{\hat{X}}_{-K_{\hat{X}},1}(t)=\pi_{f+\alpha}(t)$, i.e. $I$-series of $\hat{X}$ is equal to $\pi_{f_1}$ up to renormalization\footnote{As we will show below index of $\hat{X}$ and $X$ is equal to 2, hence renormalization is trivial: $\alpha=0$, $\widetilde{t}=t$.}. Let us compute $\pi_{f_1}$. The products of monomials $\prod_{n_i}(x^{m_i})^{n_i}$ gives a nonzero summand to the series of constant terms if $\sum n_im_i=0$; in our case put $n_{xyz}=d,\ n_x=a,\ n_y=b,\ n_z=c$. Then $n_{x^{-1}}=a+d$,  $n_{y^{-1}}=b+d$, $n_{z^{-1}}=c+d$. Hence
	\begin{equation}\label{eq: Phi f1=}
		\pi_{f_1}=\sum_{a,b,c,d\geq 0}\frac{(2a+2b+2c+4d)!}{a!b!c!d!(a+d)!(b+d)!(c+d)!}t^{2a+2b+2c+4d}=1+6t^2+114t^4+2940t^6+87570t^8+\ldots
	\end{equation}
	By \ref{c4} general anticanonical section $S'\in|-K_Y|$ is smooth. Applying the proposition \ref{prop: I-series}, we
	conclude that the restricted from $Y$ regularized $I$-series $I^{Y\to S'}_{-K_Y,1}$ for smooth anticanonical section of
	$Y$ is equal to $\pi_f$.
	
	Therefore we computed the $I$-series for the smoothing $Y$ of $X$ not using the geometry of $Y$. It is easy to check that $Y$ is a Fano variety $B_5$, because it is unique Fano threefold with invariants $(\rho,\ r,\ \deg,\ b)=(1,2,40,0)$. Since $B_5$ is a section of Grassmannian $G(2,5)$ by three hyperplanes, its $I$-series may be computed by applying the quantum Lefschetz formula \cite{Ga00} to the $I$-series of $G(2,5)$ provided in \cite{BFKS1,BCK03}:
	\[
	I_{G(2,5)}=\sum_{d\geq 0}\frac{t^d}{(d!)^2}\sum_{d\geq j_2\geq j_1\geq j_0=0}\frac{1}{(d-j_2)!\prod_{i=2}^{3}((d-j_{i-1})!(j_{i-1}-j_{i-2})!j_{i-1}!)}.
	\]
	Applying quantum Lefschetz to the $I$-series of Grassmannian $I_{G(2,5)}$ one indeed deduces is \ref{eq: Phi f1=}.
\end{ex}

\section{Generalizations}\label{sec:generalizations}

Unfortunately only half of non-toric Fano threefolds are smoothings of nodal toric. In particular the only one Fano threefold of principal series admit a small toric degeneration (it is the variety $V_{22}$). Note that for nodal toric varieties it is easy to prove projective normality and the smoothness of general anticanonical section, and one can show these properties holds for the smoothing as well (as in \ref{c2}). All the smoothings $Y$ we obtained are rational. The same method could be applied to obtain more general class of smoothings, if we consider not only the toric varieties, but also complete intersections inside them (with Gorenstein terminal singularities) --- these varieties also admit a smoothing (\ref{thm: Namikawa}), and there are similar relations between invariants of the smoothing and the degeneration, and it is not so hard to compute the cohomology of such varieties \cite{DH86}, Hilbert polynomial, and Gromov--Witten theory.

But birational class of complete intersection in toric variety is arbitrary. Many of non-degenerating to nodal toric threefolds are themselves the complete intersections in weighted projective spaces. Batyrev and Kreuzer found all nodal half-anticanonical hypersurfaces in toric fourfolds of index 2: there are around 160 of them, and 100 are cones over the toric varieties studied in this
paper, the remaining 60 cases cover almost all non-degenerating to toric Fano varieties.

Another direction for generalizations is toric varieties with arbitrary Gorenstein singularities. For a pair of nonterminal Gorenstein toric Fano threefolds $\PP(\Delta_{16})$, $\PP(\Delta_{18})$ Przyjalkowski \cite{Pr1} constructed a pair of Laurent polynomials $f_{16}$, $f_{18}$ with Newton polytopes coinciding with the corresponding fan polytopes $\Delta_{16}$, $\Delta_{18}$ such that these polynomials are Landau--Ginzburg models mirror symmetric to Fano varieties of principal series $V_{16}$ and $V_{18}$; so it is possible that toric degenerations method works for a larger class of singularities (all Gorenstein?), although we
do not know if the pairs $\PP(\Delta_{16}, V_{16})$ and $\PP(\Delta_{18}, V_{18})$ are the degenerations.

\subsection*{Acknowledgement} 
Author thanks V. Iskovskikh and V. Golyshev for the interest in this work, C. Shramov for the support during the preparation of this paper, Yu. Prokhorov and D. van Straten for useful discussions, and A. Kuznetsov for a lot of provided improvements.

% Bibliography macros
\providecommand{\arxiv}[1]{\href{http://arxiv.org/abs/#1}{arXiv:#1}}

\bigskip

{\bf Author's remark about this version of the article.}

This paper is a part of my PhD thesis titled
``Toric degenerations of Fano manifolds'' 
that was defended on April 3, 2008 at Steklov Mathematical Institute.
The text of the thesis and a brief summary (both in Russian) are available at [1] and [2].
Original version of this paper was written in Russian and is available at [3].
I translated this text to English in Fall 2008 during my visit to Johannes Gutenberg University in Mainz,
and since then first English version of this text is available at [4]. 
Unfortunately I have lost the TeX-file of the English version soon after, and all my attempts to upload PDF-file to \href{https://arxiv.org}{arXiv} failed.
In 2012 I have made a 4-page-long English version of this paper omitting all proofs, it was available at [5]. 

In 2015 Egor Yasinsky re-TeXified this paper, based on Russian TeX-file, English PDF-file and a short English TeX-file.
Since then [5] corresponds to the new version.
I am greatly indebted to Egor: if not him, this paper would not be on arXiv.

I used this opportunity to fix some of the most obvious typos and to make the layout a little bit prettier and easier to read.
In this version section ``Introduction'' is not numbered, so numeration of Sections and the respective propositions is shifted by one.
Apart from that, this version of the paper is almost the same as 2008 version.
In the passed ten years main result has found various applications and generalizations that are not mentioned in the last two sections,
I might update them in the next version. 
All the remaining mistakes and typos are due to myself.

Sergey Galkin, 2018

[1] \url{http://www.mi.ras.ru/~galkin/papers/disser.pdf}	
[2] \url{http://www.mi.ras.ru/~galkin/papers/auto.pdf}		
[3] \url{http://www.mi.ras.ru/~galkin/papers/term_def.pdf}	
[4] \url{http://www.mi.ras.ru/~galkin/work/3a.pdf}		
[5] \url{http://research.ipmu.jp/ipmu/sysimg/ipmu/893.pdf}

\end{document}